\documentclass{amsart}
\usepackage[english]{babel}
\usepackage{amsmath,amssymb,enumerate,latexsym}
\usepackage{hyperref}

\newcommand{\tmop}[1]{\ensuremath{\operatorname{#1}}}
\newtheorem{lemma}{Lemma}[section]
\newtheorem{theorem}{Theorem}[section]
\newtheorem*{conj}{Chern's Conjecture}
\numberwithin{equation}{section}

\begin{document}

\title[Rigidity of minimal submanifolds]{Rigidity of Minimal Submanifolds in Spheres of Higher Codimension}

\author{Li Lei}
\address{School of Mathematical Sciences \\ Chongqing Normal University \\ Chongqing 401131 \\ China}
\email{leili@cqnu.edu.cn}


\begin{abstract}
  We investigate Chern's conjecture for minimal submanifolds of higher 
  codimension in the unit sphere. We establish two integral inequalities for closed minimal 
  submanifolds. Then we obtain a new rigidity theorem under a pinching condition 
  on the squared norm of the second fundamental form.
\end{abstract}
\maketitle

\section{Introduction}

In the late 1960s, Simons {\cite{Simons1968}} and Chern--do Carmo--Kobayashi 
{\cite{CdK1970}} proved that if $M$ is an $n$-dimensional closed minimal
submanifold in $\mathbb{S}^{n + p}$, then the squared length of the second
fundamental form satisfies
\[ \int_M \left[ \left( 2 - \frac{1}{p} \right) S - n \right] S \geqslant 0.
\]
It follows that if $S \leqslant \frac{n}{2 - 1 / p}$, then either $S \equiv 0$
or $S \equiv \frac{n}{2 - 1 / p}$. Moreover, Chern-do Carmo-Kobayashi
 proved that the only compact minimal submanifolds satisfying
$S = \frac{n}{2 - 1 / p}$ are Clifford minimal hypersurfaces and the Veronese 
surface. Afterwards, Li--Li {\cite{LL1992}} and Chen--Xu {\cite{CX}} improved
Simons' pinching constant to $\frac{2 n}{3}$ for $p \geqslant 3$. For an
$n$-dimensional minimal submanifold $M$ in a unit sphere, it follows from the
Gauss equation that the scalar curvature of $M$ is given by $R = n (n - 1) -
S$. This implies that $S$ is a constant if and only if $R$ is a constant. In
1968 and 1970, Chern proposed the following famous conjecture
{\cite{Chern1968,CdK1970}}.

\begin{conj}
  Let $M$ be a closed minimal submanifold with constant scalar curvature in
  the unit sphere $\mathbb{S}^{n + 1}$. Then the set of all possible values of
  the scalar curvature of $M$ is discrete.
\end{conj}

Over the past four decades, substantial progress has been made on
Chern's conjecture (see {\cite{Cheng1997,GT2012,GXXZ2016,SWY,Xin}} for more
details). In 1983, Peng and Terng {\cite{PT1983a}} made a breakthrough on
Chern's conjecture. They verified that if $M^n$ is a closed minimal hypersurface
with constant scalar curvature in $\mathbb{S}^{n + 1}$, and if $S > n$, then
$S > n + \frac{1}{12 n}$. In particular, for the case $n = 3$, they proved
that if $S > 3$, then $S \geqslant 6$. Afterwards, Yang--Cheng
{\cite{YC1990,YC1994,YC1998}} improved the pinching constant $\frac{1}{12 n}$
to $\frac{n}{3}$. Later, Suh--Yang {\cite{SY2007}} improved this pinching
constant to $\frac{3 n}{7}$. In 1993, Chang {\cite{Chang1993}} solved Chern's
conjecture in dimension three. Deng--Gu--Wei {\cite{DGW2017}} proved that any
closed Willmore minimal hypersurface with constant scalar curvature in
$\mathbb{S}^5$ must be isoparametric. He--Xu--Zhao {\cite{HXZ2026}} proved that
a closed minimal hypersurface in $\mathbb{S}^5$ with constant scalar curvature
and constant third mean curvature is isoparametric. Ge--Liu--Luo--Yan
{\cite{GLLY}} and Deng--Kou {\cite{DK}} obtained corresponding rigidity results
under constant Gauss--Kronecker curvature. For closed minimal hypersurfaces in
spheres, the scalar curvature pinching phenomenon without the assumption of
constant scalar curvature was investigated by several authors
{\cite{CI1999,DX2011,LXX2017,PT1983b,WX2007,XX2017,Zhang2010}}.

In the higher codimensional situation, Calabi \cite{Calabi} classified all minimal
immersions of $\mathbb{S}^2$ with constant Gaussian curvature into
$\mathbb{S}^N$. Kozlowski--Simon {\cite{KS}} proved that if a complete minimal
surface in $\mathbb{S}^{2 + p}$ satisfies $\frac{4}{3} \leqslant S \leqslant
\frac{5}{3}$, then $S \equiv \frac{4}{3}$ or $\frac{5}{3}$. Denote by $S_2$
the second largest eigenvalue of the fundamental matrix $\langle A_{\alpha},
A_{\beta} \rangle$, Lu {\cite{Lu}} proved a rigidity theorem for compact
minimal submanifolds in the sphere satisfying $S + S_2 \leqslant n$. Recently,
Ding--Ge--Li {\cite{DGL2025,DGL2026}} made progress on Simon's third gap
conjecture for closed minimal surfaces in spheres. Ge--Li--Zhang {\cite{GLZ}}
proved a second gap theorem for compact minimal submanifolds with flat normal
bundle.

In this paper, we investigate minimal submanifolds of higher codimension in spheres. 
We establish  the following integral inequalities for
closed minimal submanifolds in spheres.

\begin{theorem}
  Let $M$ be an $n$-dimensional compact minimal submanifold in $\mathbb{S}^{n
  + p}$. Then
  \[ \int_M | \nabla^2 h |^2 \leqslant \int_M (8 S - 2 n - 3) | \nabla h |^2 .
  \]
  The equality holds if and only if at each point of $M$, either $S = 0$, or there is an orthonormal frame
  such that the components of $h$ and $\nabla h$ satisfy $h^1_{1 1} = - h^1_{2
  2} = h^2_{1 2} = h^2_{2 1},$ $h_{111}^1 = - h_{122}^1 = h_{112}^2 = -
  h_{222}^2,$ $h_{112}^1 = - h_{222}^1 = - h_{111}^2 = h_{122}^2,$ and all
  other independent components are $0$.
\end{theorem}

\begin{theorem}
  Let $M$ be an $n$-dimensional compact minimal submanifold in
  $\mathbb{S}^{n + p}$. For any positive constant $b$ such that $b \geqslant
  \frac{8}{63} (4 S - n)$, we have
  \[ \int_M \left[ \frac{7}{3 b} S^3 - \frac{1}{2} S^2 + 7 bS - \frac{2}{9}
     (21 b - 2) n \right] S \geqslant 0. \]
\end{theorem}

As a consequence of the above integral inequality, we prove a pinching theorem
for closed minimal submanifolds in spheres.

\begin{theorem}
  Let $M$ be an $n (\geqslant 3)$-dimensional compact minimal submanifold in
  $\mathbb{S}^{n + p}$. If
  \[ S \leqslant \frac{2 n}{3} + \frac{n}{500} - \frac{1}{180}, \]
  then $S \equiv 0$; equivalently, $M$ is totally geodesic.
\end{theorem}

\section{Preliminaries}

Let $M^n$ be a submanifold in $\mathbb{S}^{n + p}$. We denote by
$\overline{\nabla}$ the Levi-Civita connection of $\mathbb{S}^{n + p}$, and by
$\nabla$ the connection induced on $M$. Let $h$ denote the second fundamental
form of $M$, which is given by
\[ \overline{\nabla}_X Y = \nabla_X Y + h (X, Y) . \]
Denote by $\nabla^{\bot}$ the connection of the normal bundle $N M$. For a
normal vector field $\xi$, the Weingarten map is given by
\[ \overline{\nabla}_X \xi = \nabla^{\bot}_X \xi - A_{\xi} (X) . \]

We choose a local orthonormal frame $\{ e_i \}_{1 \leqslant i \leqslant n}$
for the tangent bundle of $M$, and $\{ \nu_{\alpha} \}_{1 \leqslant \alpha
\leqslant p}$ for the normal bundle of $M$. Let $h^{\alpha}_{i j} = \langle h
(e_i, e_j), \nu_{\alpha} \rangle$ be the components of the second fundamental
form. Let $A_{\alpha} = A_{\nu_{\alpha}}$. Thus $A_{\alpha} (e_j) = \sum_i
h^{\alpha}_{i j} e_i$. Let $S = \sum_{\alpha, i, j} (h^{\alpha}_{i j})^2$ be
the squared length of the second fundamental form.

We denote by $h_{i j k}^{\alpha}$ the covariant derivative of $h^{\alpha}_{i
j}$, denote by $h_{i j k l}^{\alpha}$ the covariant derivative of
$h^{\alpha}_{i j k}$, and so on. By the Codazzi equation, we have
$h^{\alpha}_{i j k} = h^{\alpha}_{i k j}$.

Let $R$ and $R^{\bot}$ be the curvature tensors of the tangent bundle $T M$ and
the normal bundle $N M$, respectively. We have the following Gauss and Ricci equations.
\begin{equation}
  R_{i j k l} = \delta_{i k} \delta_{j l} - \delta_{i l} \delta_{j k} +
  \sum_{\alpha} (h^{\alpha}_{i k} h^{\alpha}_{j l} - h^{\alpha}_{i l}
  h^{\alpha}_{j k}), \label{GaussE}
\end{equation}
\begin{equation}
  R^{\bot}_{\alpha \beta k l} = \sum_i (h^{\alpha}_{i k} h^{\beta}_{i l} -
  h^{\alpha}_{i l} h^{\beta}_{i k}) . \label{RicciE}
\end{equation}

Using the Ricci identity, we get
\begin{align}
  h^{\alpha}_{i j k l} = & h^{\alpha}_{i j l k} + \sum_m (h^{\alpha}_{j m}
  R_{m i k l} + h^{\alpha}_{i m} R_{m j k l}) + \sum_{\beta} h^{\beta}_{i j}
  R^{\bot}_{\beta \alpha k l} \nonumber\\
  = & h^{\alpha}_{i j l k} + \delta_{jl} h^{\alpha}_{ik} - \delta_{jk}
  h^{\alpha}_{il} + \delta_{il} h^{\alpha}_{jk} - \delta_{ik} h^{\alpha}_{jl} 
  \label{ricciid}\\
  & + \sum_{\beta, m} (h^{\alpha}_{im} h^{\beta}_{jl} h^{\beta}_{km} +
  h^{\alpha}_{jm} h^{\beta}_{il} h^{\beta}_{km} + h^{\alpha}_{lm}
  h^{\beta}_{ij} h^{\beta}_{km}  \nonumber\\
  &  - h^{\alpha}_{im} h^{\beta}_{jk} h^{\beta}_{lm} -
  h^{\alpha}_{jm} h^{\beta}_{ik} h^{\beta}_{lm} - h^{\alpha}_{km}
  h^{\beta}_{ij} h^{\beta}_{lm}) . \nonumber
\end{align}

Suppose that $M$ is minimal. Thus, $\sum_i h^{\alpha}_{i i} = 0$. Contracting
(\ref{ricciid}) we obtain
\begin{equation}
  \sum_l h^{\alpha}_{i l j l} = nh^{\alpha}_{ij} + \sum_{\beta, l, m} (2
  h^{\beta}_{il} h^{\beta}_{jm} h^{\alpha}_{lm} - h^{\beta}_{ij}
  h^{\alpha}_{lm} h^{\beta}_{lm} - h^{\beta}_{il} h^{\alpha}_{jm}
  h^{\beta}_{lm} - h^{\alpha}_{im} h^{\beta}_{jl} h^{\beta}_{lm}) .
  \label{Laph}
\end{equation}
We then have Simons' formula.
\begin{equation}
  \frac{1}{2} \Delta S = | \nabla h |^2 + n S - \sum_{\alpha, \beta} [|
  [A_{\alpha}, A_{\beta}] |^2 + \langle A_{\alpha}, A_{\beta} \rangle^2],
  \label{LapS}
\end{equation}
where $[\cdot, \cdot]$ denotes the commutator.

Taking the covariant derivative of (\ref{ricciid}) with respect to $e_l$, we
obtain
\begin{align}
  \sum_l h^{\alpha}_{i j k l l} = & \sum_l h^{\alpha}_{ijlkl} + 2
  h^{\alpha}_{ijk} \nonumber\\
  & + \sum_{\beta, l, m} (h^{\beta}_{jl} h^{\beta}_{km} h^{\alpha}_{ilm} +
  h^{\beta}_{il} h^{\beta}_{km} h^{\alpha}_{jlm} + h^{\beta}_{km}
  h^{\alpha}_{lm} h^{\beta}_{ijl}  \nonumber\\
  & + h^{\beta}_{ij} h^{\alpha}_{lm} h^{\beta}_{klm} + h^{\beta}_{il}
  h^{\alpha}_{jm} h^{\beta}_{klm} + h^{\alpha}_{im} h^{\beta}_{jl}
  h^{\beta}_{klm}  \label{hijkll}\\
  & - h^{\beta}_{jk} h^{\beta}_{lm} h^{\alpha}_{ilm} - h^{\beta}_{ik}
  h^{\beta}_{lm} h^{\alpha}_{jlm} - h^{\beta}_{ij} h^{\beta}_{lm}
  h^{\alpha}_{klm} \nonumber\\
  &  - h^{\alpha}_{km} h^{\beta}_{lm} h^{\beta}_{ijl} -
  h^{\alpha}_{jm} h^{\beta}_{lm} h^{\beta}_{ikl} - h^{\alpha}_{im}
  h^{\beta}_{lm} h^{\beta}_{jkl}) . \nonumber
\end{align}
Using the Ricci identity, we have
\begin{align}
  \sum_l h^{\alpha}_{i j l k l} = & \sum_l h^{\alpha}_{i j l l k} \nonumber\\
  & + \sum_{l, m} (h^{\alpha}_{j l m} R_{m i k l} + h^{\alpha}_{i l m} R_{m j
  k l} + h^{\alpha}_{i j m} R_{m l k l}) + \sum_{\beta, l} h^{\beta}_{i j l}
  R^{\bot}_{\beta \alpha k l} \nonumber\\
  = & \sum_l h^{\alpha}_{ijllk} + (n + 1) h^{\alpha}_{ijk} \\
  & + \sum_{\beta, l, m} (h^{\beta}_{jm} h^{\beta}_{kl} h^{\alpha}_{ilm} +
  h^{\beta}_{im} h^{\beta}_{kl} h^{\alpha}_{jlm} + h^{\beta}_{km}
  h^{\alpha}_{lm} h^{\beta}_{ijl}  \nonumber\\
  &  - h^{\beta}_{km} h^{\beta}_{lm} h^{\alpha}_{ijl} -
  h^{\beta}_{jk} h^{\beta}_{lm} h^{\alpha}_{ilm} - h^{\beta}_{ik}
  h^{\beta}_{lm} h^{\alpha}_{jlm} - h^{\alpha}_{km} h^{\beta}_{lm}
  h^{\beta}_{ijl}) . \nonumber
\end{align}
Taking the covariant derivative of (\ref{Laph}) with respect to $e_k$, we have
\begin{align}
  \sum_l h^{\alpha}_{i j l l k} = & nh^{\alpha}_{ijk} \nonumber\\
  & + \sum_{\beta, l, m} (2 h^{\beta}_{il} h^{\beta}_{jm} h^{\alpha}_{klm} +
  2 h^{\beta}_{jm} h^{\alpha}_{lm} h^{\beta}_{ikl} + 2 h^{\beta}_{il}
  h^{\alpha}_{lm} h^{\beta}_{jkm}  \nonumber\\
  & - h^{\beta}_{jl} h^{\beta}_{lm} h^{\alpha}_{ikm} - h^{\beta}_{il}
  h^{\beta}_{lm} h^{\alpha}_{jkm} - h^{\beta}_{ij} h^{\beta}_{lm}
  h^{\alpha}_{klm}  \label{hijllk}\\
  & - h^{\alpha}_{lm} h^{\beta}_{lm} h^{\beta}_{ijk} - h^{\alpha}_{jm}
  h^{\beta}_{lm} h^{\beta}_{ikl} - h^{\alpha}_{im} h^{\beta}_{lm}
  h^{\beta}_{jkl} \nonumber\\
  &  - h^{\beta}_{ij} h^{\alpha}_{lm} h^{\beta}_{klm} -
  h^{\beta}_{il} h^{\alpha}_{jm} h^{\beta}_{klm} - h^{\alpha}_{im}
  h^{\beta}_{jl} h^{\beta}_{klm}) . \nonumber
\end{align}
Combining (\ref{hijkll})-(\ref{hijllk}), we obtain
\begin{align}
  \sum_l h^{\alpha}_{i j k l l} = & (2 n + 3) h^{\alpha}_{ijk} \nonumber\\
  & + \sum_{\beta, l, m} (2 h^{\beta}_{jm} h^{\beta}_{kl} h^{\alpha}_{ilm} +
  2 h^{\beta}_{im} h^{\beta}_{kl} h^{\alpha}_{jlm} + 2 h^{\beta}_{im}
  h^{\beta}_{jl} h^{\alpha}_{klm} + 2 h^{\beta}_{km} h^{\alpha}_{lm}
  h^{\beta}_{ijl}  \nonumber\\
  & + 2 h^{\beta}_{jm} h^{\alpha}_{lm} h^{\beta}_{ikl} + 2 h^{\beta}_{im}
  h^{\alpha}_{lm} h^{\beta}_{jkl} - 2 h^{\beta}_{jk} h^{\beta}_{lm}
  h^{\alpha}_{ilm} - 2 h^{\beta}_{ik} h^{\beta}_{lm} h^{\alpha}_{jlm}
  \nonumber\\
  & - 2 h^{\beta}_{ij} h^{\beta}_{lm} h^{\alpha}_{klm} - 2 h^{\alpha}_{km}
  h^{\beta}_{lm} h^{\beta}_{ijl} - 2 h^{\alpha}_{jm} h^{\beta}_{lm}
  h^{\beta}_{ikl} - 2 h^{\alpha}_{im} h^{\beta}_{lm} h^{\beta}_{jkl}
  \nonumber\\
  &  - h^{\beta}_{km} h^{\beta}_{lm} h^{\alpha}_{ijl} -
  h^{\beta}_{jm} h^{\beta}_{lm} h^{\alpha}_{ikl} - h^{\beta}_{im}
  h^{\beta}_{lm} h^{\alpha}_{jkl} - h^{\alpha}_{lm} h^{\beta}_{lm}
  h^{\beta}_{ijk}) . \nonumber
\end{align}
We then obtain the following Peng--Terng-type formula.
\begin{align}
  \frac{1}{2} \Delta | \nabla h |^2 = & \sum_{\alpha, i, j, k, l}
  [(h^{\alpha}_{i j k l})^2 + h^{\alpha}_{i j k} h^{\alpha}_{i j k l l}]
  \nonumber\\
  = & | \nabla^2 h |^2 + (2 n + 3) | \nabla h |^2 +  \label{lapdh2}\\
  & \sum_{\alpha, \beta, i, j, k, l, m} \Big[- 6 h_{i l}^{\alpha} h_{j l}^{\beta}
  h_{i k m}^{\alpha} h_{j k m}^{\beta} + 6 h_{i l}^{\beta} h_{j l}^{\alpha}
  h_{i k m}^{\alpha} h_{j k m}^{\beta}  \nonumber\\
  & - 3 h^{\alpha}_{im} h^{\alpha}_{jm} h^{\beta}_{ilk} h^{\beta}_{jl k} -
  h^{\alpha}_{lm} h^{\beta}_{lm} h^{\alpha}_{ijk} h^{\beta}_{ijk} \nonumber\\
  &  - 6 h^{\alpha}_{jl} h^{\alpha}_{km} h^{\beta}_{ijl}
  h^{\beta}_{ikm} + 6 h^{\alpha}_{jl} h^{\alpha}_{km} h^{\beta}_{ijm}
  h^{\beta}_{ikl}\Big] . \nonumber
\end{align}

For each $\alpha, k$, define a linear map $B_{\alpha, k} : T_x M \rightarrow
T_x M$ by
\[ B_{\alpha, k} (X) = (\nabla_{e_k} A)_{\nu_{\alpha}} (X) . \]
Thus $B_{\alpha, k} (e_j) = \sum_i h^{\alpha}_{i j k} e_i$, $| \nabla h |^2 =
\sum_{\alpha, k} | B_{\alpha, k} |^2$.

Define two curvature-type tensors by
\[ \hat{R}_{i j k l} = \sum_{\alpha} (h_{i k}^{\alpha} h_{j l}^{\alpha} - h_{i
   l}^{\alpha} h_{j k}^{\alpha}), \quad \tilde{R}_{i j k l} = \sum_{\alpha, m}
   (h_{i k m}^{\alpha} h_{j l m}^{\alpha} - h_{i l m}^{\alpha} h_{j k
   m}^{\alpha}) . \]
The tensor $\hat{R}$ is the difference between the curvature tensors of the submanifold and the ambient space.

Thus, formula (\ref{lapdh2}) can be rewritten as follows.

\begin{theorem}
  Let $M$ be an $n$-dimensional minimal submanifold in $\mathbb{S}^{n + p}$.
  Then  
  \begin{align}
    \frac{1}{2} \Delta | \nabla h |^2 = & | \nabla^2 h |^2 + (2 n + 3) |
    \nabla h |^2 - 3 \langle \hat{R}, \tilde{R} \rangle \nonumber\\
    & - \sum_{\alpha, \beta, k} \Big[3 \langle A_{\alpha}^2, B_{\beta, k}^2
    \rangle + \langle A_{\alpha}, A_{\beta} \rangle \langle B_{\alpha, k},
    B_{\beta, k} \rangle   \label{lapABR}\\
    &  + 6 (\langle A_{\alpha} A_{\beta}, B_{\alpha, k} B_{\beta,
    k} \rangle - \langle A_{\alpha} A_{\beta}, B_{\beta, k} B_{\alpha, k}
    \rangle)\Big] . \nonumber
  \end{align}
\end{theorem}

\section{An upper bound for $| \nabla^2 h |$}

In this section, we will derive an estimate for the integral of $| \nabla^2 h
|$. To estimate the terms in the formula (\ref{lapABR}), we first establish some
inequalities for matrices and tensors.

\begin{lemma}
  \label{aibi}Let $\{ a_i \}_{i = 1}^n, \{ b_i \}_{i = 1}^n$ be real numbers.
  We have
  \[ \Big( \sum_i a_i b_i \Big)^2 - \sum_i a_i^2 b_i^2 - 2 \sum_{i
     \geqslant 2} a_1 a_i b_i^2 \leqslant \Big( \sum_i a_i^2 \Big) \Big(
     \frac{1}{2} b_1^2 + \frac{3}{2} \sum_{i \geqslant 2} b_i^2 \Big) . \]
\end{lemma}

\begin{proof}
  By homogeneity, we may assume
  \begin{equation}
    \frac{1}{2} b_1^2 + \frac{3}{2} \sum_{i \geqslant 2} b_i^2 = 1.
    \label{bi2}
  \end{equation}

  Set
  \[ A = \left[\begin{array}{c}
       a_2\\
       \vdots\\
       a_n
     \end{array}\right], \quad B = \left[\begin{array}{c}
       b_2\\
       \vdots\\
       b_n
     \end{array}\right], \quad C = \left[\begin{array}{c}
       b_2 (b_2 - b_1)\\
       \vdots\\
       b_n (b_n - b_1)
     \end{array}\right], \]
  \[ D = \left[\begin{array}{ccc}
       b_2^2 &  & \\
       & \ddots & \\
       &  & b_n^2
     \end{array}\right] . \]
  Regarding the difference between the right-hand side and the left-hand side
  of the inequality as a quadratic form of $a_i$, we have  
  \begin{align}
    & \sum_i a_i^2 - \Big( \sum_i a_i b_i \Big)^2 + \sum_i a_i^2 b_i^2 + 2
    \sum_{i = 2}^n a_1 a_i b_i^2 \nonumber\\
    = & \left[\begin{array}{cc}
      a_1 & A^T
    \end{array}\right] \left[\begin{array}{cc}
      1 & C^T\\
      C & I_{n - 1} + D - B B^T
    \end{array}\right] \left[\begin{array}{c}
      a_1\\
      A
    \end{array}\right] . \nonumber
  \end{align}
  
  Denote by $M$ the matrix of the above quadratic form. We need to verify that
  $M$ is positive semidefinite.
  
  Let
  \[ E = (I_{n - 1} + D)^{- 1} = \left[\begin{array}{ccc}
       (1 + b_2^2)^{- 1} &  & \\
       & \ddots & \\
       &  & (1 + b_n^2)^{- 1}
     \end{array}\right] . \]
  From (\ref{bi2}), we get
  \[ 1 - B^T E B = 1 - \sum_{i = 2}^n \frac{b_i^2}{1 + b_i^2} \geqslant 1 -
     \sum_{i = 2}^n b_i^2 > 0. \]
  So, we can use the Sherman--Morrison formula to get the inverse of $I_{n -
  1} + D - B B^T$. Let
  \[ F = (I_{n - 1} + D - B B^T)^{- 1} = E + \frac{E B B^T E}{1 - B^T E B} .
  \]
  Since $E$ is positive definite, $F$ is also positive definite.
  
  Transforming $M$ to a block diagonal matrix, we get
  \[ M = \left[\begin{array}{cc}
       1 & C^T F\\
       0 & I_{n - 1}
     \end{array}\right]  \left[\begin{array}{cc}
       1 - C^T F C & 0\\
       0 & F^{- 1}
     \end{array}\right]  \left[\begin{array}{cc}
       1 & 0\\
       F C & I_{n - 1}
     \end{array}\right], \]
  By the Cauchy--Schwarz inequality we get  
  \begin{align}
    1 - C^T F C = & 1 - C^T E C - \frac{(B^T E C)}{1 - B^T E B} \nonumber\\
    \geqslant & 1 - C^T E C - \frac{(B^T E B) (C^T E C)}{1 - B^T E B}
    \nonumber\\
    = & \frac{1 - B^T E B - C^T E C}{1 - B^T E B} . \nonumber
  \end{align}
  
  Then we have  
  \begin{align}
    & 1 - B^T E B - C^T E C \nonumber\\
    = & \sum_{i = 2}^n \frac{3 b_i^2}{2 - b_1^2} - \sum_{i = 2}^n
    \frac{b_i^2}{1 + b_i^2} - \sum_{i = 2}^n \frac{b_i^2 (b_i - b_1)^2}{1 +
    b_i^2} \nonumber\\
    = & \sum_{i = 2}^n \frac{b_i^2}{(2 - b_1^2) (1 + b_i^2)} [3 (1 + b_i^2) -
    (2 - b_1^2) - (2 - b_1^2) (b_i - b_1)^2] . \nonumber\\
    = & \sum_{i = 2}^n \frac{b_i^2}{(2 - b_1^2) (1 + b_i^2)} \left[ (1 +
    b_1^2) \left( b_i + \frac{b_1  (2 - b_1^2)}{1 + b_1^2} \right)^2 +
    \frac{(1 - 2 b_1^2)^2}{1 + b_1^2} \right] \nonumber\\
    \geqslant & 0. \nonumber
  \end{align}
  
  Therefore, $M$ is positive semidefinite. This proves the inequality.
\end{proof}

\begin{lemma}
  \label{AijBijk}Let $A_{ij}$ be a real $n \times n$ symmetric tensor and
  $B_{ijk}$ be a real $n \times n \times n$ symmetric tensor. Then
  \begin{equation}
    \sum_{i, j, k, l, m}  (A_{i k} A_{j l} - A_{i l} A_{j k})  (B_{ik m} B_{j
    lm} - B_{il m} B_{j km}) \leqslant | A |^2 | B |^2 . \label{AijBijk'}
  \end{equation}
\end{lemma}

\begin{proof}
  Since both sides of this inequality are invariant under
  orthogonal transformations, we may assume $A_{ij} = a_i \delta_{ij}$ without
  loss of generality. Then the left-hand side of the inequality equals  
  \begin{align}
    \tmop{LHS} = & 2 \sum_{i, j, m} a_i a_j B_{ii m} B_{j j m} - 2 \sum_{i, j,
    m} a_i a_j B_{ijm}^2 \nonumber\\
    = & 2 \sum_{i, j, m} a_i a_j B_{ii m} B_{j j m} - 2 \sum_{i, m} a_i^2
    B_{ii m}^2 - 4 \sum_i \sum_{m \neq i} a_i a_m B_{imm}^2 \nonumber\\
    & - 4 \sum_{i, j, k \tmop{distinct}} (a_i a_j + a_j a_k + a_k a_i)
    B_{ijk}^2 \nonumber\\
    = & 2 \sum_m \left[ \Big( \sum_i a_i B_{ii m} \Big)^2 - \sum_i a_i^2
    B_{ii m}^2 - 2 \sum_{i \neq m} a_i a_m B_{iim}^2 \right] \nonumber\\
    & - 4 \sum_{i, j, k \tmop{distinct}} (a_i a_j + a_j a_k + a_k a_i)
    B_{ijk}^2 . \nonumber
  \end{align}
  
  For fixed $m$, it follows from Lemma \ref{aibi} that
  \[ \Big( \sum_i a_i B_{ii m} \Big)^2 - \sum_i a_i^2 B_{ii m}^2 - 2
     \sum_{i \neq m} a_i a_m B_{iim}^2 \leqslant | A |^2 \Big( \frac{1}{2}
     B_{m m m}^2 + \frac{3}{2} \sum_{i \neq m} B_{i i m}^2 \Big) . \]
  For distinct $i, j, k$, we have
  \[ - 2 (a_i a_j + a_j a_k + a_k a_i) = a_i^2 + a_j^2 + a_k^2 - (a_i + a_j +
     a_k)^2 \leqslant | A |^2 . \]
  Hence,  
  \begin{align}
    \tmop{LHS} \leqslant & \sum_m | A |^2 \Big( B_{m m m}^2 + 3 \sum_{i \neq
    m} B_{i i m}^2 \Big) + 2 | A |^2 \sum_{i, j, k \tmop{distinct}}
    B_{ijk}^2 \nonumber\\
    \leqslant & | A |^2 \Big( \sum_m B_{m m m}^2 + 3 \sum_m \sum_{i \neq m}
    B_{i i m}^2 + 6 \sum_{i, j, k \tmop{distinct}} B_{ijk}^2 \Big) = | A |^2
    | B |^2 . \nonumber
  \end{align}
  
\end{proof}

\begin{lemma}
  \label{uvxy}Let $u, v, x, y$ be four vectors in an inner product space $V$.
  Then
  \[ \langle u, v \rangle \langle x, y \rangle + \langle u, x \rangle \langle
     v, y \rangle - \langle u, y \rangle \langle v, x \rangle \leqslant \frac{| u
     |^2 | y |^2 + | v |^2 | x |^2}{2} . \]
  The equality holds if and only if $u, v, x, y$ are in a $2$-dimensional
  subspace, $| u | | y | = | v | | x |$, and the oriented angle from $u$ to
  $v$ equals the oriented angle from $x$ to $y$ in the 2-plane.
\end{lemma}

\begin{proof}
  In $\bigwedge^2 (V)$, the inner product of bivectors is given by
  \[ \langle u \wedge v, x \wedge y \rangle = \langle u, x \rangle \langle v,
     y \rangle - \langle u, y \rangle \langle v, x \rangle . \]
  Thus we have  
  \begin{align}
    & \langle u, v \rangle \langle x, y \rangle + \langle u, x \rangle
    \langle v, y \rangle - \langle u, y \rangle \langle v, x \rangle
    \nonumber\\
    = & \langle u, v \rangle \langle x, y \rangle + \langle u \wedge v, x
    \wedge y \rangle \nonumber\\
    \leqslant & \langle u, v \rangle \langle x, y \rangle + | u \wedge v |  |
    x \wedge y | . \nonumber
  \end{align}
  
  Let $\theta \in [0, \pi]$ be the angle between $u$ and $v$, and $\varphi \in
  [0, \pi]$ be the angle between $x$ and $y$. Then  
  \begin{align}
    & \langle u, v \rangle \langle x, y \rangle + | u \wedge v |  | x \wedge
    y | \nonumber\\
    = & | u | | v | | x | | y | \cos \theta \cos \varphi + | u | | v | | x | |
    y | \sin \theta \sin \varphi \nonumber\\
    = & | u | | v | | x | | y | \cos (\theta - \varphi) \nonumber\\
    \leqslant & | u | | v | | x | | y | . \nonumber
  \end{align}
  
  Therefore, we obtain
  \[ \langle u, v \rangle \langle x, y \rangle + \langle u, x \rangle \langle
     v, y \rangle - \langle u, y \rangle \langle v, x \rangle \leqslant | u | | v |
     | x | | y | \leqslant \frac{| u |^2 | y |^2 + | v |^2 | x |^2}{2} . \]
  The equality holds if and only if $u \wedge v, x \wedge y$ are in the same
  direction, $\theta = \varphi$ and $| u | | y | = | v | | x |$.
\end{proof}

Lu {\cite{Lu}} and Ge--Tang {\cite{GT2008}} proved the DDVV inequality.

\begin{lemma}
  Let $A_1, \ldots, A_p$ be $p$ real $n \times n$ symmetric matrices. Then
  \[ \sum_{r, s = 1}^p \|[A_r, A_s]\|^2 \leqslant \left( \sum_{r = 1}^p \|A_r \|^2
     \right)^2 . \]
  The equality holds if and only if, up to an orthonormal action,
  \[ A_1 = \mathrm{diag} (\mu, - \mu, 0, \ldots), \quad A_2 = \mathrm{diag}
     \left( \left(\begin{array}{cc}
       0 & \mu\\
       \mu & 0
     \end{array}\right), 0, \ldots \right), \quad A_3 = \cdots = A_p = 0. \]
\end{lemma}

We next prove the following lemma.

\begin{lemma}
  \label{A2B2+1/2}Let $\{ A_r \}_{r = 1}^p$, $\{ B_r \}_{r = 1}^p$ be real $n
  \times n$ symmetric matrices. Then\\
  (i)
  \[ \sum_{r, s} (\langle A_r A_s, B_r B_s \rangle - \langle A_r A_s, B_s B_r
     \rangle) \leqslant \frac{1}{2} (\sum_r \| A_r \|^2)  (\sum_r \| B_r \|^2)
     ; \]
  (ii)
  \[ \sum_{r, s} (\langle A_r^2, B_s^2 \rangle + \langle A_r A_s, B_r B_s
     \rangle - \langle A_r A_s, B_s B_r \rangle) \leqslant (\sum_r \| A_r
     \|^2)  (\sum_r \| B_r \|^2) ; \]
  (iii)
  \[ \sum_{r, s} (\langle A_r, A_s \rangle \langle B_r, B_s \rangle + \langle
     A_r A_s, B_r B_s \rangle - \langle A_r A_s, B_s B_r \rangle) \leqslant
     (\sum_r \| A_r \|^2)  (\sum_r \| B_r \|^2) . \]
\end{lemma}

\begin{proof}
  Applying the Cauchy--Schwarz inequality and the DDVV inequality, we obtain  
  \begin{align}
    & \sum_{r, s} (\langle A_r A_s, B_r B_s \rangle - \langle A_r A_s, B_s
    B_r \rangle) \nonumber\\
    = & \frac{1}{2} \sum_{r, s}  \langle [A_r, A_s], [B_r, B_s] \rangle
    \nonumber\\
    \leqslant & \frac{1}{2} \sum_{r, s} \| [A_r, A_s] \|  \| [B_r, B_s] \|
    \nonumber\\
    \leqslant & \frac{1}{2} \Big( \sum_{r, s} \| [A_r, A_s] \|^2
    \Big)^{\frac{1}{2}} \Big( \sum_{r, s} \| [B_r, B_s] \|^2
    \Big)^{\frac{1}{2}} \nonumber\\
    \leqslant & \frac{1}{2} \Big( \sum_r \| A_r \|^2 \Big) \Big( \sum_r \|
    B_r \|^2 \Big) . \nonumber
  \end{align}
  
  Thus, we obtain conclusion (i).
  
  Denote by $a^r_{i j}$ and $b^r_{i j}$ the entries of $A_r, B_r$
  respectively. For $1 \leqslant i, j, k, l \leqslant n$, define vectors in
  $\mathbb{R}^p$ as
  \[ \xi_{i j} = (a^1_{i j}, \ldots, a^p_{i j}), \qquad \eta_{i j} = (b^1_{i
     j}, \ldots, b^p_{i j}) . \]
  Then we have  
  \begin{align}
    \sum_{r, s} \langle A_r^2, B_s^2 \rangle = & \sum_{r, s, i, j, k, l}
    a^r_{i k} a^r_{j k} b^s_{i l} b^s_{j l} = \sum_{i, j, k, l} \langle \xi_{i
    k}, \xi_{j k} \rangle \langle \eta_{i l}, \eta_{j l} \rangle \nonumber
  \end{align}
  
  and  
  \begin{align}
    & \sum_{r, s} \langle A_r A_s, B_r B_s \rangle - \langle A_r A_s, B_s B_r
    \rangle \nonumber\\
    = & \sum_{r, s, i, j, k, l} (a^r_{i k} a^s_{j k} b^r_{i l} b^s_{j l} -
    a^r_{i k} a^s_{j k} b^s_{i l} b^r_{j l}) \nonumber\\
    = & \sum_{i, j, k, l} (\langle \xi_{i k}, \eta_{i l} \rangle \langle
    \xi_{j k}, \eta_{j l} \rangle - \langle \xi_{i k}, \eta_{j l} \rangle
    \langle \xi_{j k}, \eta_{i l} \rangle) . \nonumber
  \end{align}
  
  Applying Lemma \ref{uvxy}, we obtain  
  \begin{align}
    & \sum_{r, s} (\langle A_r^2, B_s^2 \rangle + \langle A_r A_s, B_r B_s
    \rangle - \langle A_r A_s, B_s B_r \rangle) \nonumber\\
    \leqslant & \sum_{i, j, k, l} \frac{1}{2} (| \xi_{i k} |^2 | \eta_{j l}
    |^2 + | \xi_{j k} |^2 | \eta_{i l} |^2) \nonumber\\
    = & \Big(\sum_r \| A_r \|^2\Big)  \Big(\sum_r \| B_r \|^2\Big) . \nonumber
  \end{align}
  
  Therefore, we obtain conclusion (ii).
  
  Let
  \[ A_r = \left[\begin{array}{ccc}
       \alpha_{r, 1} & \cdots & \alpha_{r, n}
     \end{array}\right], \quad B_r = \left[\begin{array}{ccc}
       \beta_{r, 1} & \cdots & \beta_{r, n}
     \end{array}\right], \]
  where $\alpha_{r, i}, \beta_{r, i}$ are column vectors. Then we have
  \[ \sum_{r, s} \langle A_r, A_s \rangle \langle B_r, B_s \rangle = \sum_{r,
     s, i, j} \langle \alpha_{r, i}, \alpha_{s, i} \rangle \langle \beta_{r,
     j}, \beta_{s, j} \rangle, \]
  and  
  \begin{align}
    & \sum_{r, s} \langle A_r A_s, B_r B_s \rangle - \langle A_r A_s, B_s B_r
    \rangle \nonumber\\
    = & \sum_{r, s, i, j} \langle \alpha_{r, i}, \beta_{r, j} \rangle \langle
    \alpha_{s, i}, \beta_{s, j} \rangle - \langle \alpha_{r, i}, \beta_{s, j}
    \rangle \langle \alpha_{s, i}, \beta_{r, j} \rangle . \nonumber
  \end{align}

  Using Lemma \ref{uvxy}, we get  
  \begin{align}
    & \sum_{r, s} (\langle A_r, A_s \rangle \langle B_r, B_s \rangle +
    \langle A_r A_s, B_r B_s \rangle - \langle A_r A_s, B_s B_r \rangle)
    \nonumber\\
    \leqslant & \sum_{r, s, i, j} (\langle \alpha_{r, i}, \alpha_{s, i}
    \rangle \langle \beta_{r, j}, \beta_{s, j} \rangle + \langle \alpha_{r,
    i}, \beta_{r, j} \rangle \langle \alpha_{s, i}, \beta_{s, j} \rangle -
    \langle \alpha_{r, i}, \beta_{s, j} \rangle \langle \alpha_{s, i},
    \beta_{r, j} \rangle) \nonumber\\
    \leqslant & \sum_{r, s, i, j} \frac{1}{2} (| \alpha_{r, i} |^2 | \beta_{s,
    j} |^2 + | \alpha_{s, i} |^2 | \beta_{r, j} |^2) \nonumber\\
    = & \sum_{r, s} \| A_r \|^2  \| B_s \|^2 . \nonumber
  \end{align}
  
  Thus we obtain the conclusion (iii).
\end{proof}

Now we can estimate the terms in formula (\ref{lapABR}).

\begin{theorem}
  Let $M$ be an $n$-dimensional minimal submanifold in
  $\mathbb{S}^{n + p}$. Then
  \begin{equation}
    \frac{1}{2} \Delta | \nabla h |^2 \geqslant | \nabla^2 h |^2 + (2 n + 3 -
    8 S) | \nabla h |^2 . \label{LapDh2<gtr>=}
  \end{equation}
  The equality holds if and only if $S = 0$ or there is an orthonormal frame
  such that the components of $h$ and $\nabla h$ satisfy
  \begin{equation}
    \begin{array}{l}
      h^1_{1 1} = - h^1_{2 2} = h^2_{1 2} = h^2_{2 1},\\
      h_{111}^1 = - h_{122}^1 = h_{112}^2 = - h_{222}^2,\\
      h_{112}^1 = - h_{222}^1 = - h_{111}^2 = h_{122}^2,
    \end{array} \label{=cond}
  \end{equation}
  and all other independent components are $0$.
\end{theorem}

\begin{proof}
  Applying Lemma \ref{AijBijk}, we obtain
  \[ \langle \hat{R}, \tilde{R} \rangle = \sum_{\alpha, \beta, i, j, k, l, m}
     (h_{i k}^{\alpha} h_{j l}^{\alpha} - h_{i l}^{\alpha} h_{j k}^{\alpha})
     (h_{i k m}^{\beta} h_{j l m}^{\beta} - h_{i l m}^{\beta} h_{j k
     m}^{\beta}) \leqslant S | \nabla h |^2 . \]
  From Lemma \ref{A2B2+1/2}, we have
  \[ \sum_{k, \alpha, \beta} (\langle A_{\alpha} A_{\beta}, B_{\alpha, k}
     B_{\beta, k} \rangle - \langle A_{\alpha} A_{\beta}, B_{\beta, k}
     B_{\alpha, k} \rangle) \leqslant \frac{1}{2} S | \nabla h |^2, \]
  \[ \sum_k \sum_{\alpha, \beta} (\langle A_{\alpha}^2, B_{\beta, k}^2 \rangle
     + \langle A_{\alpha} A_{\beta}, B_{\alpha, k} B_{\beta, k} \rangle -
     \langle A_{\alpha} A_{\beta}, B_{\beta, k} B_{\alpha, k} \rangle)
     \leqslant S | \nabla h |^2 \]
  and
  \[ \sum_k \sum_{\alpha, \beta} (\langle A_{\alpha}, A_{\beta} \rangle
     \langle B_{\alpha, k}, B_{\beta, k} \rangle + \langle A_{\alpha}
     A_{\beta}, B_{\alpha, k} B_{\beta, k} \rangle - \langle A_{\alpha}
     A_{\beta}, B_{\beta, k} B_{\alpha, k} \rangle) \leqslant S | \nabla h |^2
     . \]
  Substituting these inequalities into formula (\ref{lapABR}), we obtain
  \[ \frac{1}{2} \Delta | \nabla h |^2 \geqslant | \nabla^2 h |^2 + (2 n + 3 -
     8 S) | \nabla h |^2 . \]

  When the equality in (\ref{LapDh2<gtr>=}) holds, by the equality condition
  of the DDVV inequality, we can choose an orthonormal frame such that
  \[ h^1_{1 1} = - h^1_{2 2} = h^2_{1 2} = h^2_{2 1}, \]
  and all other independent components of $h$ vanish. If $h^{\alpha}_{i j}$ are not
  all zero, from the equality condition in Lemma \ref{uvxy}, we get
  $h^{\alpha}_{i j k} = 0$ if $\alpha > 2$ or $i > 2$ and
  \[ \begin{array}{l}
       h_{111}^1 = - h_{122}^1 = h_{112}^2 = - h_{222}^2,\\
       h_{112}^1 = - h_{222}^1 = - h_{111}^2 = h_{122}^2 .
     \end{array} \]

  Conversely, these conditions imply equality in (\ref{LapDh2<gtr>=}).
  Hence the coefficient 8 is sharp.
\end{proof}

Now we obtain the following integral inequality for $| \nabla^2 h |^2$.

\begin{theorem}\label{intddh2}
  \label{intLapDh2}Let $M$ be an $n$-dimensional compact minimal submanifold
  in $\mathbb{S}^{n + p}$. Then
  \[ \int_M | \nabla^2 h |^2 \leqslant \int_M (8 S - 2 n - 3) | \nabla h |^2 .
  \]
  The equality holds if and only if condition $(\ref{=cond})$ holds
  everywhere on $M$.
\end{theorem}

\section{Estimates for the eigenvalues of $\sigma_{\alpha \beta}$}

In {\cite{Lu}}, Lu proved the following inequalities.

\begin{lemma}
  \label{maxrij}Suppose $\lambda_1, \ldots, \lambda_n$ are real numbers and
  \[ \lambda_1 + \cdots + \lambda_n = 0. \]
  Let $r_{ij}$ be nonnegative numbers for $i < j$. Then we have
  \[ \sum_{i < j} (\lambda_i - \lambda_j)^2 r_{ij} \leqslant \Big( \sum_i
     \lambda_i^2 \Big) \Big( \sum_{i < j} r_{ij} + \max_{i < j} \{r_{ij} \}
     \Big) . \]
\end{lemma}

\begin{lemma}
  \label{Luineq}Let $A_1, A_2, \ldots, A_p$ be symmetric $n \times n$ matrices
  which satisfy
  \[ \langle A_i, A_j \rangle = 0 \quad \tmop{if} \enspace 2 \leqslant i < j,
     \qquad \| A_2 \| \geqslant \cdots \geqslant \| A_p \| . \]
  Then
  \[ \sum_{j = 2}^p \| [A_1, A_j] \|^2 \leqslant \| A_1 \|^2 \Big( \sum_{j =
     2}^p \| A_j \|^2 + \| A_2 \|^2 \Big) . \]
\end{lemma}

\

Using these inequalities, Lu {\cite{Lu}} gave a new proof of Li--Li's matrix
inequality \cite{LL1992}. In the following, we use Lemma \ref{Luineq} to obtain a refined
Li--Li matrix inequality.

\begin{lemma}
  \label{3/2Ai22}Let $A_1, A_2, \ldots, A_p$ be symmetric matrices which
  satisfy
  \[ \langle A_i, A_j \rangle = 0 \quad \tmop{if} \enspace 2 \leqslant i < j,
     \qquad \| A_2 \| \geqslant \cdots \geqslant \| A_p \| . \]
  Then
  \[ \sum_{i, j} \| [A_i, A_j] \|^2 + \sum_i \| A_i \|^4 \leqslant
     \frac{3}{2} \Big( \sum_i \| A_i \|^2 \Big)^2 - \frac{1}{2} \Big(
     \sum_i \| A_i \|^2 - 2 \| A_2 \|^2 \Big)^2 . \]
\end{lemma}

\begin{proof}
  Using Lemma \ref{Luineq}, we obtain  
  \begin{align}
    \sum_{i < j} \| [A_i, A_j] \|^2 = & \sum_{i = 1}^{p - 1} \sum_{j = i +
    1}^p \| [A_i, A_j] \|^2 \nonumber\\
    \leqslant & \sum_{i = 1}^{p - 1} \| A_i \|^2 \Big( \sum_{j = i + 1}^p \|
    A_j \|^2 + \| A_{i + 1} \|^2 \Big)  \label{AiAj2}\\
    = & \sum_{i < j} \| A_i \|^2 \| A_j \|^2 + \sum_{i = 1}^{p - 1} \| A_i
    \|^2 \| A_{i + 1} \|^2 \nonumber\\
    = & \frac{1}{2} \Big( \sum_i \| A_i \|^2 \Big)^2 - \frac{1}{2} \sum_i
    \| A_i \|^4 + \sum_{i = 1}^{p - 1} \| A_i \|^2 \| A_{i + 1} \|^2 .
    \nonumber
  \end{align}
  
  Note that  
  \begin{align}
    & \sum_{i = 1}^{p - 1} \| A_i \|^2 \| A_{i + 1} \|^2 \nonumber\\
    \leqslant & \| A_1 \|^2 \| A_2 \|^2 + \sum_{i = 3}^p \| A_2 \|^2 \| A_i
    \|^2  \label{A1A2A3}\\
    = & \frac{1}{4} \Big( \sum_i \| A_i \|^2 \Big)^2 - \frac{1}{4} \Big(
    \sum_i \| A_i \|^2 - 2 \| A_2 \|^2 \Big)^2 . \nonumber
  \end{align}

  Combining (\ref{AiAj2}) and (\ref{A1A2A3}), we obtain
  \[ \sum_{i < j} \| [A_i, A_j] \|^2 \leqslant \frac{3}{4} \Big( \sum_i \|
     A_i \|^2 \Big)^2 - \frac{1}{2} \sum_i \| A_i \|^4 - \frac{1}{4} \Big(
     \sum_i \| A_i \|^2 - 2 \| A_2 \|^2 \Big)^2 . \]
  This proves the result.
\end{proof}

\begin{lemma}
  \label{3/2Ai22'}Let $A_1, A_2, \ldots, A_p$ be symmetric matrices which
  satisfy
  \[ \langle A_i, A_j \rangle = 0 \quad \tmop{if} \enspace 2 \leqslant i < j,
     \qquad \| A_2 \| \geqslant \cdots \geqslant \| A_p \| . \]
  Then  
  \begin{align}
    \sum_{i, j} \| [A_i, A_j] \|^2 & + \sum_i \| A_i \|^4 \leqslant
    \frac{3}{2} \Big( \sum_i \| A_i \|^2 \Big)^2 \nonumber\\
    & + 2 \| [A_1, A_2] \|^2 - 4 \| A_1 \|^2 \| A_2 \|^2 + 2 \| A_2 \|^2 \|
    A_3 \|^2 . \nonumber
  \end{align}
\end{lemma}

\begin{proof}
  Using Lemma \ref{Luineq} again, we obtain  
  \begin{align}
    & \sum_{i < j} \| [A_i, A_j] \|^2 - \| [A_1, A_2] \|^2 \nonumber\\
    = & \sum_{j = 3}^p \| [A_1, A_j] \|^2 + \sum_{i = 2}^p \sum_{j = i + 1}^p
    \| [A_i, A_j] \|^2 \nonumber\\
    \leqslant & \| A_1 \|^2 \Big( \sum_{j = 3}^p \| A_j \|^2 + \| A_3 \|^2
    \Big) + \sum_{i = 2}^{p - 1} \| A_i \|^2 \Big( \sum_{j = i + 1}^p \|
    A_j \|^2 + \| A_{i + 1} \|^2 \Big) \nonumber\\
    = & - 2 \| A_1 \|^2 \| A_2 \|^2 + \| A_1 \|^2 \| A_3 \|^2 + \sum_{i < j}
    \| A_i \|^2 \| A_j \|^2 + \sum_{i = 1}^{p - 1} \| A_i \|^2 \| A_{i + 1}
    \|^2  \label{AiAj2b}\\
    = & - 2 \| A_1 \|^2 \| A_2 \|^2 + \frac{1}{2} \Big( \sum_i \| A_i \|^2
    \Big)^2 - \frac{1}{2} \sum_i \| A_i \|^4 \nonumber\\
    & + \| A_1 \|^2 \| A_3 \|^2 + \sum_{i = 1}^{p - 1} \| A_i \|^2 \| A_{i +
    1} \|^2 . \nonumber
  \end{align}
  
  We have  
  \begin{align}
    & \| A_1 \|^2 \| A_3 \|^2 + \sum_{i = 1}^{p - 1} \| A_i \|^2 \| A_{i + 1}
    \|^2 \nonumber\\
    \leqslant & \| A_1 \|^2 \| A_3 \|^2 + \| A_1 \|^2 \| A_2 \|^2 + \| A_2
    \|^2 \| A_3 \|^2 + \sum_{i = 3}^{p - 1} \| A_1 \|^2 \| A_{i + 1} \|^2
    \nonumber\\
    = & \| A_2 \|^2 \| A_3 \|^2 + \| A_1 \|^2  \sum_{i = 2}^{p - 1} \| A_{i +
    1} \|^2  \label{A1A2A3b}\\
    \leqslant & \| A_2 \|^2 \| A_3 \|^2 + \frac{1}{4}  \Big( \sum_{i = 1}^{p
    - 1} \| A_{i + 1} \|^2 \Big)^2 . \nonumber
  \end{align}
  
  Combining (\ref{AiAj2b}) and (\ref{A1A2A3b}), we obtain  
  \begin{align}
    \sum_{i < j} \| [A_i, A_j] \|^2 - \| [A_1, A_2] \|^2 \leqslant & \| A_2
    \|^2 \| A_3 \|^2 - 2 \| A_1 \|^2 \| A_2 \|^2 \nonumber\\
    & + \frac{3}{4} \Big( \sum_i \| A_i \|^2 \Big)^2 - \frac{1}{2} \sum_i
    \| A_i \|^4 . \nonumber
  \end{align}
  
  This proves the result.
\end{proof}

Let $\sigma_{\alpha \beta} = \langle A_{\alpha}, A_{\beta} \rangle$, which is
a $p \times p$ symmetric matrix. Let $\{ S_{\alpha} \}_{1 \leqslant \alpha
\leqslant p}$ be the eigenvalues of $\sigma_{\alpha \beta}$ in descending
order, i.e.
\[ S_1 \geqslant S_2 \geqslant \cdots \geqslant S_p . \]
Then $S = \sum_{\alpha} S_{\alpha}$. Choose an orthonormal frame $\{
\nu_{\alpha} \}$ for $N_x M$ such that $\sigma_{\alpha \alpha} = S_{\alpha}$
and $\sigma_{\alpha \beta} = 0$ if $\alpha \neq \beta$. Set
\[ E = \frac{3}{2} S^2 - \sum_{\alpha, \beta} | [A_{\alpha}, A_{\beta}] |^2 -
   \sum_{\alpha} S_{\alpha}^2 . \]
From Lemmas \ref{3/2Ai22} and \ref{3/2Ai22'}, we get
\begin{equation}
  (S - 2 S_2)^2 \leqslant 2 E \label{S-2S2<less>2E}
\end{equation}
and
\begin{equation}
  2 S_1 S_2 - | [A_1, A_2] |^2 \leqslant \frac{1}{2} E + S_2 S_3 \leqslant
  \frac{1}{2} E + \frac{1}{2} S S_3 . \label{<less>E+2S2S3}
\end{equation}
Then we have the following estimates for $S_{\alpha}$.

\begin{theorem}
  \label{S2<gtr>E}Let $M$ be a submanifold in $\mathbb{S}^{n + p}$. Then
  \[ \frac{S}{2} + \sqrt{\frac{E}{2}} \geqslant S_1 \geqslant S_2 \geqslant
     \frac{S}{2} - \sqrt{\frac{E}{2}}, \qquad \sum_{\alpha > 2} S_{\alpha}
     \leqslant \sqrt{2 E} . \]
\end{theorem}

\begin{proof}
  It follows from (\ref{S-2S2<less>2E}) that
  \[ S - 2 S_2 \leqslant \sqrt{2 E} . \]
  Then we have
  \[ S_2 = \frac{1}{2} [S - (S - 2 S_2)] \geqslant \frac{S}{2} -
     \sqrt{\frac{E}{2}}, \]
  \[ S_1 \leqslant S - S_2 = \frac{1}{2} [S + (S - 2 S_2)] \leqslant
     \frac{S}{2} + \sqrt{\frac{E}{2}} \]
  and
  \[ \sum_{\alpha > 2} S_{\alpha} = S - S_1 - S_2 \leqslant S - 2 S_2
     \leqslant \sqrt{2 E} . \]
  
\end{proof}

\section{A lower bound for $| \nabla^2 h |^2$}

In this section we derive a lower bound for $| \nabla^2 h |$. Since
$h^{\alpha}_{i j k l}$ is symmetric on indices $i, j, k$, we get
\begin{align}
  0 \leqslant & \sum_{\alpha, i, j, k, l} (h^{\alpha}_{ijkl} + h^{\alpha}_{l
  ijk} + h^{\alpha}_{klij} + h^{\alpha}_{jk li})^2 \nonumber\\
  = & 4 | \nabla^2 h |^2 + 12 \sum_{\alpha, i, j, k, l} h^{\alpha}_{ijkl}
  h^{\alpha}_{ijl k} \nonumber\\
  = & 4 | \nabla^2 h |^2 + 6 \sum_{\alpha, i, j, k, l} [ (h^{\alpha}_{ijkl})^2
  + (h^{\alpha}_{ijl k})^2 - (h^{\alpha}_{ijkl} - h^{\alpha}_{ijl k})^2]
  \nonumber\\
  = & 16 | \nabla^2 h |^2 - 6 \sum_{\alpha, i, j, k, l}  (h^{\alpha}_{ijkl} -
  h^{\alpha}_{ijl k})^2 . \nonumber
\end{align}

So,
\[ | \nabla^2 h |^2 \geqslant \frac{3}{8} \sum_{\alpha, i, j, k, l} 
   (h^{\alpha}_{ijkl} - h^{\alpha}_{ijl k})^2 . \]
Using the Ricci identity, we get
\begin{align}
  h^{\alpha}_{ijkl} - h^{\alpha}_{ijl k} = & \sum_m (h^{\alpha}_{i m} R_{m j k
  l} + h^{\alpha}_{j m} R_{m i k l}) - \sum_{\beta} h^{\beta}_{i j}
  R^{\bot}_{\alpha \beta k l} \nonumber\\
  = & h^{\alpha}_{i k} \delta_{j l} - h^{\alpha}_{i l} \delta_{j k} +
  h^{\alpha}_{j l} \delta_{i k} - h^{\alpha}_{j k} \delta_{i l} \nonumber\\
  & + \sum_m (h^{\alpha}_{i m}  \hat{R}_{m j k l} + h^{\alpha}_{j m} 
  \hat{R}_{m i k l}) - \sum_{\beta} h^{\beta}_{i j} R^{\bot}_{\alpha \beta k
  l} . \nonumber
\end{align}

Define two tensors $P, Q$ by
\[ P^{\alpha}_{i j k l} = h^{\alpha}_{i k} \delta_{j l} - h^{\alpha}_{i l}
   \delta_{j k} + h^{\alpha}_{j k} \delta_{i l} - h^{\alpha}_{j l} \delta_{i
   k}, \]
\[ Q^{\alpha}_{i j k l} = - \sum_m (h^{\alpha}_{i m}  \hat{R}_{m j k l} +
   h^{\alpha}_{j m}  \hat{R}_{m i k l}) + \sum_{\beta} h^{\beta}_{i j}
   R^{\bot}_{\alpha \beta k l} . \]
Then
\begin{equation}
  | \nabla^2 h |^2 \geqslant \frac{3}{8} | P - Q |^2 = \frac{3}{8} | P |^2 -
  \frac{3}{4} \langle P, Q \rangle + \frac{3}{8} | Q |^2 . \label{P-Q}
\end{equation}
By direct computations, we get
\begin{equation}
  | P |^2 = 4 n S, \qquad \langle P, Q \rangle = 4 \sum_{\alpha, \beta} [|
  [A_{\alpha}, A_{\beta}] |^2 + \langle A_{\alpha}, A_{\beta} \rangle^2] = 6
  S^2 - 4 E. \label{PQ}
\end{equation}

We derive the following estimate for $| Q |$.

\begin{lemma}
  \label{ddh2<gtr>}We have
  \[ | Q |^2 \geqslant \frac{9}{2} S^3 - 42 \sqrt{2 E} S^2 . \]
\end{lemma}

\begin{proof}
  If $\frac{9}{2} S \leqslant 42 \sqrt{2 E}$, the assertion is valid. If
  $\frac{9}{2} S > 42 \sqrt{2 E}$, let
  \[ \varepsilon = \frac{\sqrt{2 E}}{S}, \]
  then $0 \leqslant \varepsilon < \frac{3}{28}$.
  
  From Theorem \ref{S2<gtr>E}, we have
  \begin{equation}
    S_1 \geqslant S_2 \geqslant \frac{1 - \varepsilon}{2} S, \qquad
    \sum_{\alpha > 2} S_{\alpha} \leqslant \varepsilon S. \label{S2<gtr>}
  \end{equation}
  Then by (\ref{<less>E+2S2S3}), we get
  \begin{equation}
    2 S_1 S_2 - | [A_1, A_2] |^2 \leqslant \frac{\varepsilon^2 + 2
    \varepsilon}{4} S^2 . \label{S1S2-}
  \end{equation}

  Now we choose an orthonormal frame $\{ e_i \}$ for $T_x M$ such that $h^1_{i
  j}$ are diagonal matrices. Let $\lambda_i = h^1_{i i}$, which satisfy
  $\sum_i \lambda_i = 0$, $\sum_i \lambda_i^2 = S_1$. Without loss of
  generality, we assume
  \[ | h^2_{1 2} | = \max_{i \neq j} \{ | h^2_{i j} | \} . \]
  Set
  \[ U = (h^2_{12})^2, \quad V = \sum_i (h^2_{ii})^2, \quad W = \sum_{i < j, j
     \geqslant 3} (h^2_{ij})^2 . \]
  Then we have $S_2 = V + 2 U + 2 W$.
  
  From Lemma \ref{maxrij}, we have
  \[ \frac{1}{2} | [A_1, A_2] |^2 = \sum_{i < j} (\lambda_i - \lambda_j)^2
     (h^2_{ij})^2 \leqslant S_1  (W + 2 U) . \]
  This yields
  \[ S_1 S_2 - \frac{1}{2} | [A_1, A_2] |^2 \geqslant S_1  (V + W) . \]
  By (\ref{S2<gtr>}) and (\ref{S1S2-}), we get
  \[ V + W \leqslant \frac{2 S_1 S_2 - | [A_1, A_2] |^2}{2 S_1} \leqslant
     \frac{\varepsilon^2 + 2 \varepsilon}{4 (1 - \varepsilon)} S = \kappa S,
  \]
  where $\kappa = \frac{\varepsilon^2 + 2 \varepsilon}{4 (1 - \varepsilon)}$.
  Then we have
  \begin{equation}
    U \geqslant U - \frac{V}{4} \geqslant \frac{S_2}{2} - V - W \geqslant
    \frac{1 - \varepsilon - 4 \kappa}{4} S. \label{h212}
  \end{equation}
  On the other hand, we have  
  \begin{align}
    & S_1 S_2 - \frac{1}{2} | [A_1, A_2] |^2 \nonumber\\
    = & S_1 S_2 - \sum_{i < j} (\lambda_i - \lambda_j)^2 (h^2_{ij})^2
    \nonumber\\
    = & S_1 S_2 - (\lambda_1 - \lambda_2)^2 (h^2_{12})^2 - \sum_{i < j, j
    \geqslant 3} (\lambda_i - \lambda_j)^2 (h^2_{ij})^2 \nonumber\\
    \geqslant & S_1 S_2 - (\lambda_1 - \lambda_2)^2 (h^2_{12})^2 - \sum_{i <
    j, j \geqslant 3} 2 (\lambda_i^2 + \lambda_j^2) (h^2_{ij})^2 \nonumber\\
    \geqslant & 2 S_1 (U + W) - (\lambda_1 - \lambda_2)^2 U - 2 S_1 W
    \nonumber\\
    = & (2 S_1 - (\lambda_1 - \lambda_2)^2) U. \nonumber
  \end{align}
  
  From (\ref{S1S2-}) and (\ref{h212}), we get
  \[ 2 S_1 - (\lambda_1 - \lambda_2)^2 \leqslant \frac{2 S_1 S_2 - | [A_1,
     A_2] |^2}{2 U} \leqslant \frac{\varepsilon^2 + 2 \varepsilon}{2 (1 -
     \varepsilon - 4 \kappa)} S = \eta S, \]
  where $\eta = \frac{\varepsilon^2 + 2 \varepsilon}{2 (1 - \varepsilon - 4
  \kappa)}$. Hence we have
  \begin{equation}
    (\lambda_1 - \lambda_2)^2 \geqslant 2 S_1 - \eta S \geqslant (1 -
    \varepsilon - \eta) S, \label{la1-la2}
  \end{equation}
  \begin{equation}
    - 2 \lambda_1 \lambda_2 = - \lambda_1^2 - \lambda_2^2 + (\lambda_1 -
    \lambda_2)^2 \geqslant S_1 - [2 S_1 - (\lambda_1 - \lambda_2)^2] \geqslant
    \frac{1 - \varepsilon - 2 \eta}{2} S. \label{-2la1la2}
  \end{equation}

  We have  
  \begin{align}
    Q^1_{1212} = & - \sum_m (h^1_{1 m}  \hat{R}_{m 2 1 2} + h^1_{2 m} 
    \hat{R}_{m 1 1 2}) - \sum_{\alpha} h^{\alpha}_{1 2} R^{\bot}_{\alpha 1 1
    2} \nonumber\\
    = & (\lambda_1 - \lambda_2) \sum_{\alpha} (2 (h^{\alpha}_{1 2})^2 - h_{1
    1}^{\alpha} h_{2 2}^{\alpha}) . \nonumber
  \end{align}
  
  By (\ref{S2<gtr>}), (\ref{h212}) and (\ref{-2la1la2}), we get  
  \begin{align}
    & \sum_{\alpha} (2 (h^{\alpha}_{12})^2 - h_{11}^{\alpha} h_{22}^{\alpha})
    \nonumber\\
    = & - \lambda_1 \lambda_2 + 2 (h^2_{12})^2 - h_{11}^2 h_{22}^2 +
    \sum_{\alpha > 2} (2 (h^{\alpha}_{12})^2 - h_{11}^{\alpha}
    h_{22}^{\alpha}) \nonumber\\
    \geqslant & - \lambda_1 \lambda_2 + 2 (h^2_{12})^2 - \frac{(h_{11}^2)^2 +
    (h_{22}^2)^2}{2} - \sum_{\alpha > 2} \frac{(h_{11}^{\alpha})^2 +
    (h_{22}^{\alpha})^2}{2} \nonumber\\
    \geqslant & - \lambda_1 \lambda_2 + 2 \left( U - \frac{V}{4} \right) -
    \frac{1}{2} \sum_{\alpha > 2} S_{\alpha} \nonumber\\
    \geqslant & \left( \frac{1 - \varepsilon - 2 \eta}{4} + \frac{1 -
    \varepsilon - 4 \kappa}{2} - \frac{1}{2} \varepsilon \right) S \nonumber\\
    = & \left( \frac{3}{4} - \frac{5 \varepsilon}{4} - 2 \kappa -
    \frac{\eta}{2} \right) S. \nonumber
  \end{align}
  
  Combining the above inequality with (\ref{la1-la2}), we obtain
  \[ (Q^1_{1212})^2 \geqslant (1 - \varepsilon - \eta) \left( \frac{3}{4} -
     \frac{5 \varepsilon}{4} - 2 \kappa - \frac{\eta}{2} \right)^2 S^3 =
     \varphi S^3, \]
  where
  \[ \varphi = (1 - \varepsilon - \eta) \left( \frac{3}{4} - \frac{5
     \varepsilon}{4} - 2 \kappa - \frac{\eta}{2} \right)^2 . \]
  Note that $Q^{\alpha}_{i j k l} = Q^{\alpha}_{j i k l} = - Q^{\alpha}_{i j l
  k}$. So,
  \[ \sum_{i, j, k, l} (Q^1_{i j k l})^2 \geqslant 4 (Q^1_{1212})^2 \geqslant
     4 \varphi S^3 . \]

  Now we choose an orthonormal frame $\{e_i \}$ for $T_x M$ such that
  $h^2_{ij}$ is a diagonal matrix. Using the same argument as above, we
  obtain
  \[ \sum_{i, j, k, l} (Q^2_{i j k l})^2 \geqslant 4 \varphi S^3 . \]
  Thus,
  \[ | Q |^2 \geqslant \sum_{i, j, k, l} ((Q^1_{i j k l})^2 + (Q^2_{i j k
     l})^2) \geqslant 8 \varphi S^3 . \]
  By a direct calculation, we get
  \[ \varphi \geqslant \frac{9}{16} - \frac{21 \varepsilon}{4} \quad
     \tmop{for} \enspace 0 \leqslant \varepsilon < \frac{3}{28} . \]
  Therefore, we have
  \[ | Q |^2 \geqslant \left( \frac{9}{2} - 42 \varepsilon \right) S^3 =
     \frac{9}{2} S^3 - 42 \sqrt{2 E} S^2 . \]
  
\end{proof}

This lemma together with (\ref{P-Q}) and (\ref{PQ}) yields

\begin{theorem}
  \label{ddh2<gtr>S3}Let $M$ be a minimal submanifold in $\mathbb{S}^{n +
  p}$. Then
  \[ | \nabla^2 h |^2 \geqslant \frac{3}{2} n S - \frac{9}{2} S^2 +
     \frac{27}{16} S^3 + 3 E - \frac{63}{4}  \sqrt{2 E} S^2 . \]
\end{theorem}

\section{A pinching theorem}

Now we derive the following integral inequality for compact minimal
submanifolds in spheres.

\begin{theorem}
  \label{21b-2}Let $M$ be an $n$-dimensional compact minimal submanifold in
  $\mathbb{S}^{n + p}$. For any positive constant $b$ such that $b \geqslant
  \frac{8}{63} (4 S - n)$, we have
  \[ \int_M \left[ \frac{7}{3 b} S^3 - \frac{1}{2} S^2 + 7 bS - \frac{2}{9}
     (21 b - 2) n \right] S \geqslant 0. \]
\end{theorem}

\begin{proof}
  From Theorem \ref{intddh2}, we have
  \[ \int_M | \nabla^2 h |^2 \leqslant \int_M (8 S - 2 n - 3) | \nabla h |^2
     \leqslant \int_M \left( \frac{63 b}{4} - 3 \right) | \nabla h |^2 . \]
  Using Theorem \ref{ddh2<gtr>S3}, we have  
  \begin{align}
    | \nabla^2 h |^2 \geqslant & \frac{3}{2} n S - \frac{9}{2} S^2 +
    \frac{27}{16} S^3 + 3 E - \frac{63}{4}  \sqrt{2 E} S^2 \nonumber\\
    \geqslant & \frac{3}{2} n S - \frac{9}{2} S^2 + \frac{27}{16} S^3 + 3 E -
    \frac{63 b}{4} E - \frac{63}{8 b} S^4 . \nonumber
  \end{align}
  
  Combining these estimates, we obtain
  \[ \int_M \left( \frac{63 b}{4} - 3 \right) | \nabla h |^2 \geqslant \int_M
     \left[ \frac{3}{2} n S - \frac{9}{2} S^2 + \frac{27}{16} S^3 + 3 E -
     \frac{63 b}{4} E - \frac{63}{8 b} S^4 \right] . \]
  Equivalently,
  \[ 0 \geqslant \int_M \left[ \frac{3}{2} n S - \frac{9}{2} S^2 +
     \frac{27}{16} S^3 - \frac{63}{8 b} S^4 + \left( 3 - \frac{63 b}{4}
     \right)  (E + | \nabla h |^2) \right] . \]
  By (\ref{LapS}) we get
  \[ \int_M (| \nabla h |^2 + E) = \int_M \left( \frac{3}{2} S^2 - n S \right)
     . \]
  Therefore, we obtain  
  \begin{align}
    0 \geqslant & \int_M \left[ \frac{3}{2} n S - \frac{9}{2} S^2 +
    \frac{27}{16} S^3 - \frac{63}{8 b} S^4 + \left( 3 - \frac{63 b}{4} \right)
    \left( \frac{3}{2} S^2 - n S \right) \right] \nonumber\\
    = & \int_M \frac{27}{8} \left[ \frac{2}{9} (21 b - 2) n - 7 bS +
    \frac{1}{2} S^2 - \frac{7}{3 b} S^3 \right] S. \nonumber
  \end{align}
  
  \ 
\end{proof}

As a consequence, we obtain a pinching theorem.

\begin{theorem}
  Let $M$ be an $n (\geqslant 3)$-dimensional compact minimal submanifold in
  $\mathbb{S}^{n + p}$. If $S \leqslant \frac{2 n}{3} + \frac{n}{500} -
  \frac{1}{180}$, then $S \equiv 0$.
\end{theorem}

\begin{proof}
  Let
  \[ b = \frac{6 n^2}{n - 2} . \]
  Since $b \geqslant \frac{8}{63} (4 S - n)$, the integral inequality in Theorem
  \ref{21b-2} is valid.
  
  Define a function
  \[ f (s) = \frac{7}{3 b} s^3 - \frac{1}{2} s^2 + 7 bs - \frac{2}{9}  (21 b -
     2) n. \]
  Differentiating, we obtain
  \[ f' (s) = \frac{7}{b} s^2 - s + 7 b. \]
  Thus $f' (s) > 0$ if $s < 7 b$. By a direct calculation, we get $f (\frac{1003
  n}{1500} - \frac{1}{180}) < 0$. Hence $f (s) < 0$ for all $s \leqslant
  \frac{1003 n}{1500} - \frac{1}{180}$.
  
  From Theorem \ref{21b-2}, we have
  \[ \int_M f (S) S \geqslant 0. \]
  Hence we obtain $S \equiv 0$.
\end{proof}


\begin{thebibliography}{10}
  \bibitem{Chang1993}S. P. Chang, On minimal hypersurfaces with constant
  scalar curvatures in $S^4$, {\em{J. Differential Geom.}},
  \textbf{37}(1993), 523-534.
  
  \bibitem{Calabi}E. Calabi, Minimal immersions of surfaces in Euclidean
  spheres, {\em{J. Differential Geom.}}, \textbf{1}(1967), 111-125.
  
  \bibitem{CX}Q. Chen and S. L. Xu, Rigidity of compact minimal submanifolds
  in a unit sphere, {\em{Geom. Dedicata}}, \textbf{45}(1993), 83-88.
  
  \bibitem{CI1999}Q. M. Cheng and S. Ishikawa, A characterization of the
  Clifford torus, {\em{Proc. Amer. Math. Soc.}}, \textbf{127}(1999),
  819-828.
  
  \bibitem{Cheng1997}S. Y. Cheng, On the Chern conjecture for minimal
  hypersurface with constant scalar curvatures in the spheres, Tsing Hua
  Lectures on Geometry and Analysis, International Press, Cambridge, MA, 1997,
  59-78.
  
  \bibitem{Chern1968}S. S. Chern, Minimal submanifolds in a Riemannian
  manifold. University of Kansas, Department of Mathematics Technical Report
  19, Univ. of Kansas, Lawrence, Kan., 1968.
  
  \bibitem{CdK1970}S. S. Chern, M. do Carmo and S. Kobayashi, Minimal
  submanifolds of a sphere with second fundamental form of constant length,
  {\em{Functional Analysis and Related Fields}}, Springer-Verlag, Berlin,
  1970, 59-75.
  
  \bibitem{DGW2017}Q. T. Deng, H. L. Gu and Q. Y. Wei, Closed Willmore
  minimal hypersurfaces with constant scalar curvature in $\mathbb{S}^5 (1)$
  are isoparametric, {\em{Adv. Math.}}, \textbf{314}(2017), 278-305.
  
  \bibitem{DK}Q. T. Deng and Y. J. Kou, Closed Minimal Hypersurfaces in
  $\mathbb{S}^5 (1)$ with Constant Scalar and Gauss-Kronecker Curvatures,
  arXiv:2607.06588, 2026.
  
  \bibitem{DX2011}Q. Ding and Y. L. Xin, On Chern's problem for rigidity of
  minimal hypersurfaces in the spheres, {\em{Adv. Math.}},
  \textbf{227}(2011), 131-145.
  
  \bibitem{DGL2025}W. R. Ding, J. Q. Ge and F. G. Li, Pinching rigidity of
  minimal surfaces in spheres, {\em{Sci. China Math.}}, \textbf{68}(2025),
  2189-2206.
  
  \bibitem{DGL2026}W. R. Ding, J. Q. Ge and F. G. Li, On Simon's third gap
  conjecture for minimal surfaces in spheres, arXiv:2603.03070v3, 2026.
  
  \bibitem{GLZ}J. Q. Ge, F. G. Li, and Y. H. Zhang, On Chern's conjecture
  for minimal submanifolds with flat normal bundle in spheres, preprint, 2026,
  arXiv:2607.10733, 2026.
  
  \bibitem{GLLY}J. Q. Ge, T. Liu, K. Y. Luo and W. J. Yan, Rigidity of
  closed minimal hypersurfaces in $\mathbb{S}^5$, arXiv:2606.29246, 2026.
  
  \bibitem{GT2008}J. Q. Ge and Z. Z. Tang, A proof of the DDVV conjecture
  and its equality case, {\em{Pacific J. Math.}}, \textbf{237}(2008),
  87-95.
  
  \bibitem{GT2012}J. Q. Ge and Z. Z. Tang, Chern conjecture and
  isoparametric hypersurfaces, Differential geometry, Adv. Lect. Math., 22,
  International Press, Somerville, MA, 2012, 49-60.
  
  \bibitem{GXXZ2016}J. R. Gu, H. W. Xu, Z. Y. Xu and E. T. Zhao, A survey on
  rigidity problems in geometry and topology of submanifolds, Proceedings of
  the 6th International Congress of Chinese Mathematicians, {\em{Adv. Lect.
  Math.}}, \textbf{37}, Higher Education Press \& International Press,
  Beijing-Boston, 2016, 79-99.
  
  \bibitem{HXZ2026}C. C. He, H. W. Xu and E. T. Zhao, Classification of
  closed minimal hypersurfaces with constant scalar curvature in
  $\mathbb{S}^5$, arXiv:2603.01181, 2026.
  
  \bibitem{KS} M. Kozlowski and U. Simon, Minimal immersions of 2-manifolds
  into spheres, {\em{Math. Z.}}, \textbf{186}(1984), 377-382.
  
  \bibitem{LXX2017}L. Lei, H. W. Xu and Z. Y. Xu, On Chern's conjecture for
  minimal hypersurfaces in spheres, arXiv:1712.01175v1.
  
  \bibitem{LL1992}A. M. Li and J. M. Li, An intrinsic rigidity theorem for
  minimal submanifolds in a sphere, {\em{Arch. Math. (Basel)}},
  \textbf{58}(1992), 582-594.
  
  \bibitem{Lu}Z. Q. Lu, Normal scalar curvature conjecture and its
  applications, {\em{J. Funct. Anal.}} \textbf{261} (2011), 1284-1308.
  
 
  \bibitem{PT1983a}C. K. Peng and C. L. Terng, Minimal hypersurfaces of
  sphere with constant scalar curvature, {\em{Ann. of Math. Stud.}},
  \textbf{103}, Princeton Univ. Press, Princeton, NJ, 1983, 177-198.
  
  \bibitem{PT1983b}C. K. Peng and C. L. Terng, The scalar curvature of
  minimal hypersurfaces in spheres, {\em{Math. Ann.}}, \textbf{266}(1983),
  105-113.
  
  \bibitem{SWY}M. Scherfner, S. Weiss and S. T. Yau, A review of the Chern
  conjecture for isoparametric hypersurfaces in spheres, Advances in Geometric
  Analysis, {\em{Adv. Lect. Math.}}, \textbf{21}, International Press,
  Somerville, MA, 2012, 175-187.
  
  \bibitem{Simons1968}J. Simons, Minimal varieties in Riemannian manifolds,
  {\em{Ann. of Math.}}, \textbf{88}(1968), 62-105.
  
  \bibitem{SY2007}Y. J. Suh and H. Y. Yang, The scalar curvature of minimal
  hypersurfaces in a unit sphere, {\em{Commun. Contemp. Math.}},
  \textbf{9}(2007), 183-200.
  
  \bibitem{WX2007}S. M. Wei and H. W. Xu, Scalar curvature of minimal
  hypersurfaces in a sphere, {\em{Math. Res. Lett.}}, \textbf{14}(2007),
  423-432.
  
  \bibitem{Xin}Y. L. Xin, Minimal submanifolds and related topics (Second
  Edition), Nankai Tracts in Mathematics, 16, World Scientific Publishing Co.,
  Inc., River Edge, NJ, 2018.
  
  
  \bibitem{XX2017}H. W. Xu and Z. Y. Xu, On Chern's conjecture for minimal
  hypersurfaces and rigidity of self-shrinkers, {\em{J. Funct. Anal.}},
  \textbf{273}(2017), 3406-3425.
  
  
  \bibitem{YC1990}H. C. Yang and Q. M. Cheng, A note on the pinching
  constant of minimal hypersurfaces with constant scalar curvature in the unit
  sphere, {\em{Kexue Tongbao}}, \textbf{35}(1990), 167-170; {\em{Chinese
  Sci. Bull.}}, \textbf{36}(1991), 1-6.
  
  \bibitem{YC1994}H. C. Yang and Q. M. Cheng, An estimate of the pinching
  constant of minimal hypersurfaces with constant scalar curvature in the unit
  sphere, {\em{Manuscripta Math.}}, \textbf{84}(1994), 89-100.
  
  \bibitem{YC1998}H. C. Yang and Q. M. Cheng, Chern's conjecture on minimal
  hypersurfaces, {\em{Math. Z.}}, \textbf{227}(1998), 377-390.
  
  
  \bibitem{Zhang2010}Q. Zhang, The pinching constant of minimal
  hypersurfaces in the unit spheres, {\em{Proc. Amer. Math. Soc.}},
  \textbf{138}(2010), 1833-1841.
\end{thebibliography}
\end{document}